\documentclass[11pt]{article}
\usepackage[latin1]{inputenc}
\usepackage{ amsmath, amsfonts, amssymb,amsthm,amscd}\input amssym.def
\usepackage{graphicx}
\def\sqw{\hbox{\rlap{\leavevmode\raise.3ex\hbox{$\sqcap$}}$%
\sqcup$}}

 \def\br{\hbox{\it I\hskip -2pt R}} \def\bn{\hbox{\it 
I\hskip -2pt N}} \def\bz{\hbox{\it Z\hskip -4pt Z}}  \def\bc{\hbox{\it l\hskip -5.5pt C\/}}   
 \newtheorem{theorem}{Theorem} 
\newtheorem{lemma}{Lemma} \newtheorem{proposition}{Proposition} 
\newtheorem{definition}{Definition} \newtheorem{example}{Example} 
\newtheorem{remark}{Remark}  
\newtheorem{corollary}{Corollary}

\def\dim{{\rm \ dim}\,}

\def\card{{\rm \ card\,} }

\def\reg{{\rm \ reg}\, }

\def\fz{\mathbb{Z}}

\def\fc{\mathbb{C}}

\newcommand{\Nbb}{{\mathbb N}}

\newcommand{\Zbb}{{\mathbb Z}}
\newcommand{\Cbb}{{\mathbb C}}

\newcommand{\bp}{{\mathbb P}}

\newcommand{\Gcal}{{\mathcal G}}
\newcommand{\Fcal}{{\mathcal F}}

\newcommand{\Jcal}{{\mathcal J}}

\newcommand{\Acal}{{\mathcal A}}

\newcommand{\Scal}{{\mathcal S}}
\newcommand{\Pcal}{{\mathcal P}}

\newcommand{\Mcal}{{\mathcal M}}

\begin{document}

\begin{center}
\uppercase{{\bf  Segre embeddings, Hilbert series,and  
Simon Newcomb's problem}}
\end{center}
\advance\baselineskip-3pt
\vspace{2\baselineskip}
\begin{center}
{{\sc Marcel
Morales}\\
{\small Universit\'e de Grenoble I, Institut Fourier, 
UMR 5582, B.P.74,\\
38402 Saint-Martin D'H\`eres Cedex, (FRANCE)\\
and IUFM, Université de Lyon 1, 5 rue Anselme,\\ 69317 Lyon Cedex (FRANCE)}\\
 }
\vspace{\baselineskip}
\end{center}
\noindent{\bf Abstract} {\footnote{ {\it{ Partially supported by VIAS, Hanoi, Vietnam. 


 Key words and phrases: Segre varieties;  Hilbert Series; Simon Newcomb's numbers; Grobner basis;  Betti numbers}}.}  
Monomial ideals and toric rings are closely related.  By consider a Grobner basis we can always associated to any ideal $I$ in a polynomial ring a monomial ideal ${\rm in}_\prec I$,
 in some special situations the monomial ideal ${\rm in}_\prec I$ is square free. On the other hand given any monomial ideal $I$ of a polynomial ring $S$,
 we can define the toric $K[I]\subset S$.  In this paper we will study   toric rings defined by Segre embeddings, 
we will prove that their $h-$ vectors coincides with the so called Simon Newcomb number's in probabilities and combinatorics.
 We solve the original question of
 Simon Newcomb by given a formula for the  Simon Newcomb's numbers involving only positive integer numbers.
\section{Introduction} 

First we recall the Simon Newcomb's problem: 

Consider a deck 
$\Scal$ 
of $N$ cards containing $b_i$ cards of face value $i$, 
for $i=1,...,n$ so that 
$b_1+...+b_n=N$.
Turn over the top card and put it down face up. Turn over the second card and place it on top of the first card if the face value of the second card is less than or equal to the first - otherwise start a new pile. Continue this process until all 
$N$ cards have been turned over. The number of possible cases with $k$ piles or rises so formed yields the Simon Newcomb number 
$A([{\bf b}],k)$, with specification 
$[{\bf b}]=[b_1,...,b_n]$.

We can also interpret Simon Newcomb's problem in the following way: We consider the sequence or  multiset 
$\displaystyle {\underbrace {1...1}_{b_1 }}\displaystyle {\underbrace {2...2}_{b_2 }}...\displaystyle {\underbrace {n...n}_{b_n }}$ and all the permutations of this multiset,
 that is the set $\Scal$ of all sequences 
$(u_k)_{1\leq k\leq N}$ of length $N:=b_1+ b_2+ ...+ b_n$ with $b_i$ times the symbol $i$. We say that a sequence $(u_k)$ has a descent in $k$  if $u_k>u_{k+1}$, 
the Simon Newcomb numbers $A([{\bf b}],k)$ counts the numbers of sequences with $k$ descents. In the special case where $b_1=...=b_n=1$ the Simon Newcomb numbers 
 are known as the the Eulerian numbers. 

Now, we go to the algebraic setting. Let $I$ be any ideal in a polynomial ring $S$, and let $\prec $ be a term order.
Let ${\rm in}_\prec I$ be the initial ideal with respect to this order, the initial complex 
$\Delta_\prec (I)$ of $I$ with respect to $\prec$
is the simplicial complex whose Stanley-Reisner ideal is the radical of 
${\rm in}_\prec I$. 
Now we consider a homogeneous toric ideal  $I_\Acal$ given by a finite set 
$ \Acal \subset \bn^n$, 
we can associated with it the rational polyhedron  ${\rm conv}(\Acal)$ the convex  hull of $\Acal$. 
In \cite{stu} it is showed that $\Delta_\prec (I_\Acal)$ is a regular triangulation  of ${\rm conv}(\Acal)$ and this triangulation is unimodular
 if and only if the ideal ${\rm in}_\prec (I_\Acal)$ is square free, in this case the toric ring $S/I_\Acal$ is projectively normal, 
in particular  is an arithmetically  Cohen-Macaulay ring and the multiplicity of the ring $S/I_\Acal$, which equals the volume of ${\rm conv}(\Acal)$ is given by the number of facets of  $\Delta_\prec (I_\Acal)$. 

In this paper we consider the case of toric ideals, associated to  the hypercube $\Mcal= \{0,...,b_1\}\times \{0,...,b_2\}\times...\times  \{0,...,b_n\}\subset \bn^n$,
  the corresponding toric variety is the Segre embedding of $\bp^{b_1}\times... \times \bp^{b_n} $ in  $\bp^{(b_1+1)\times (b_2+1)\times ...\times (b_n+1)-1}$.
 Blum \cite{bl} has showed  that there exists some term order (without exhibiting one), such that the toric ideal  $I_\Mcal$ has a quadratic Groebner basis
 and ${\rm in}_\prec (I_\Mcal)$ is square free. In \cite{ha} such Groebner basis was given and the fact that it is a Groebner basis was proved by direct computations. 
Our purpose is to exhibit such a Groebner basis
 in a more conceptual way using ideas developed in \cite{stu}, as a consequence we can describe the facets of the initial complex $\Delta_\prec (I_\Mcal)$, as the rises 
in the Simon Newcomb's problem.
 We get the Hilbert Poincar\'e  series of the toric ring $S/I_\Mcal$, 
and  we prove that  the $h$-vector of  the Hilbert Poincar\'e  series of the toric ring $S/I_\Mcal$ has a nice interpretation in terms of combinatorics and probability: 
 namely $h_k=A([{\bf b}],k)$, where $A([{\bf b}],k)$ are the Simon Newcomb's numbers. This allows us  to solve the original question of
 Simon Newcomb by given a formula for the  Simon Newcomb's numbers involving only positive integer numbers.

 \section{Hilbert's Series} \label{section1}  
 We start with a general lemma which allows to compute the $h-$polynomial of a generating series of a sequence.
\begin{lemma}\label{h-vectorbyhilbert}
Let $(a_l)_{l\in \Zbb }$ be a sequence of complex numbers, such that $a_l=0$ for  $l<\sigma $, set :
$ f(t)= \sum_{l\geq \sigma }a_lt^l,$ suppose that $f(t)=\frac{h(t)}{(1-t)^d}$ with $h(t)\in \Cbb [t,t^{-1}]$, 
 $h(t)=h_\sigma t^\sigma +....+h_0+h_1t+...+h_rt^r$, then
\begin{equation}\label{h-vector-hilbertfunction1}\mbox {for  \ }k=\sigma ,...,r; \ \   h_k= \sum_{j=0}^{k-\sigma }(-1)^j {{d}\choose{j}} a_{k-j} .\end{equation}

\end{lemma}
 \begin{proof}
It follows immediately by direct computation from the equality:
$$(\sum_{l\geq \sigma }a_lt^l)\times (\sum_{j=0}^d (-1)^j {{d}\choose{j}} t^j)= h_\sigma t^\sigma +....+h_0+h_1t+...+h_rt^r.$$

\end{proof}
By Theorem \ref{h-vectorbyhilbert} we  can compute the $h-$polynomial of a quotient ring $S/J$ in terms of the Hilbert function $H_{S/J}(j)$.
\begin{lemma} Let $S$ be a ring of polynomial in a finite set of variables over a field $K$, with the standard graduation, let $J\subset S$ be a  homogeneous ideal. Let  
$H_{S/J}(j)$ be the Hilbert function of $S/J$ and $\displaystyle  P_{S/J}(t)=\frac{h_0+h_1t+...+h_rt^r }{(1-t)^{d}} $, be the Hilbert-Poincaré series, where $d=\dim (S/J)$. Then
\begin{equation}\label{h-vector-hilbertfunction2}\mbox {For  \ }k=0 ,...,r; \ \ h_k= \sum_{j=0}^k (-1)^j {{d}\choose{j}} H_{S/J}(k-j) .\end{equation}
\end{lemma}

Let recall the following fact:
\begin{theorem}\label{hilbertfunction-hpolynomial}
Let $(a_l)_{l\in \Zbb}$ be a sequence of complex numbers, such that $a_l=0$ for  $l<<0$, set :
$ f(t)= \sum_{l\in \Zbb}a_lt^l,$
TFAE:
\begin{itemize}
\item There exists $h(t)\in \Cbb [t,t^{-1}]$ and a natural integer $d$ such that $f(t)=\frac{h(t)}{(1-t)^d}$.
\item There exists $\Phi (t)\in \Cbb [t,t^{-1}]$ of degree $d-1$ with leading coefficient $e_0/(d-1)! $, such that $\Phi (l)=a_l$ for $l$ large enough.
\end{itemize}
Moreover $h(1)=e_0$.
 \end{theorem} 
 Euler defined the Euler polynomials $A_n(t)$  by the following equality:
 $$\sum_{l\geq 0 }(l+1)^nt^l= \frac{A_n(t)}{(1-t)^{n+1}},$$ we can deduce from it the following identity:
 $$\sum_{l\geq 1 }l^nt^l= \frac{tA_n(t)}{(1-t)^{n+1}},$$
 the Eulerian numbers $A(n,m)$ are the coefficients of $A_n(t)$, $A_n(t)=\sum_{m=0 }^n A(n,m)t^m$. It follows from Theorem \ref{h-vectorbyhilbert} that we have:
 $$\mbox {for  \ }m=0 ,...,n ; \ \ A(n,m)= \sum_{j=0}^{k }(-1)^j {{d}\choose{j}} (m+1-j)^n  .$$
 Now we give a statement that improves slightly \cite{di}[Theorem 5.1], we also give a new proof, 
 \begin{theorem}
 Fix two integers $d,n\in \Nbb^*$. Let $(a_l)_{l\in \Zbb}$ be a sequence of complex numbers, such that $a_l=0$ for  $l<<0$, set :
$$ f(t)= \sum_{l\in \Zbb}a_lt^l,\ \ f^{<n>}(t)= \sum_{l\in \Zbb}a_{nl}t^l.$$ If $\displaystyle f(t)=\frac{h(t)}{(1-t)^{d+1}}$ with $h(t)\in \Cbb [t,t^{-1}],h(1)\not=0$ then 
$\displaystyle  f^{<n>}(t)=\frac{h^{<n>}(t)}{(1-t)^{d+1}}$ with $h^{<n>}(t)\in \Cbb [t,t^{-1}]$. Moreover $$\lim_{n\rightarrow \infty }(d!/h(1) )\frac{h^{<n>}(t)}{n^d}= tA_d(t),$$
 where $A_d(t)$ is the Eulerian polynomial
 
\end{theorem} 
\begin{proof} By Theorem \ref{hilbertfunction-hpolynomial}, there exists a polynomial $\Phi (t)$ of degree $d$ with leading coefficient $h(1)/d!$ 
such that $\Phi (l)=a_l $ for $l$ large enough, which implies obviously that $\Phi (nl)=a_{nl} $ for $l$ large enough, and by applying again 
Theorem \ref{hilbertfunction-hpolynomial} there exists a polynomial $h^{<n>}(t)\in \Cbb [t,t^{-1}]$ such that $\displaystyle f^{<n>}(t)=\frac{h^{<n>}(t)}{(1-t)^{d+1}}$.

Now we prove the second claim. Note that for $n$ large enough $a_{nl}=0,$ for all integer $l<0$. so that $f^{<n>}(t)= \sum_{l\in \Nbb}a_{nl}t^l$, hence  
by Lemma \ref{h-vectorbyhilbert}  we have $h^{<n>}_0=  a_0 $, and for all
$k\geq 1,$ $$h^{<n>}_k=\sum_{j=0}^{k }(-1)^j {{d+1}\choose{j}} a_{n(k-j)} =(-1)^k {{d+1}\choose{k}}a_0 + \sum_{j=0}^{k-1 }(-1)^j {{d+1}\choose{j}} a_{n(k-j)}.$$ 
For $n$ large enough and $k-j\not=0$ we have  $a_{n(k-j)}=\Phi (n(k-j))=h(1) (n(k-j))^d/d!+....,$ so that 
$$\frac{(d!/h(1) )h_k}{n^d}=\frac{(d!/h(1) )(-1)^k {{d+1}\choose{k}}a_0}{n^d} + \sum_{j=0}^{k-1 }\frac{(d!/h(1) )(-1)^j {{d+1}\choose{j}} h(1) (n(k-j))^d/d!+....} {n^d}.$$
By taking the limit when $n\rightarrow \infty$, we get $$\lim_{n\rightarrow \infty } \frac{(d!/h(1))h_k}{n^d}=
\sum_{j=0}^{k-1 }(-1)^j {{d+1}\choose{j}} (k-j)^d=A(d,k-1).$$
Our claim is over.
\end{proof}
\section{Segre embeddings}
\begin{definition}
 Let $\Mcal= \{0,...,b_1\}\times \{0,...,b_2\}\times...\times  \{0,...,b_n\}\subset \bn^n$. Let $\Jcal_\Mcal$ be   the kernel of the ring homomorphism 
$$f : k[T_v |v\in \Mcal] \longrightarrow  k[x_{1,0}, ...,x_{1,b_1},x_{2,0}, ...,x_{2,b_2},...,x_{n,0}, ...,x_{n,b_n}]$$ given by 
$$T_v\mapsto x_{1,v_1}x_{2,v_2}...x_{n,v_n}.$$  
In all this paper we will set $[{\bf b}]=[b_1,...,b_n]$. The projective toric variety defined by the homogeneous prime ideal $\Jcal_\Mcal\subset k[T_v |v\in \Mcal]$
 is the Segre embedding $\Sigma _{[{\bf b}]}$ of
 $\bp^{b_1}\times... \times \bp^{b_n} $ in  $\bp^{(b_1+1)\times (b_2+1)\times ...\times (b_n+1)-1}$. We denote the ideal $\Jcal_\Mcal\subset k[T_v |v\in \Mcal]$ by 
$\Jcal_{[{\bf b}]},$ and the ring $k[T_v |v\in \Mcal]/\Jcal_\Mcal\subset k[T_v |v\in \Mcal]$ by $R_{[{\bf b}]}.$
\end{definition}
We recall some known properties : 
\begin{enumerate}
\item $\dim \Sigma _{[{\bf b}]}=b_1+b_2+...+b_n$.
\item $\deg \Sigma _{[{\bf b}]}=\frac{(b_1+b_2+...+b_n)!}{b_1!b_2!...b_n!}$.
\item $H_{R_{[{\bf b}]}}(l)={{b_1+l}\choose{b_1}}\times {{b_2+l}\choose{b_2}}\times ... \times {{b_n+l}\choose{b_n}} $.
\end{enumerate}

Now by using the Hilbert function, we will prove a recurrence formula for the $h-$vector of $R_{[{\bf b}]}.$
 \begin{theorem}\label{poincarerecurrence} Let $i$ such that   $b_i>0$, we set ${[{\bf b- \varepsilon_i}]}:= {b_1,b_2,...,b_i-1,...,b_n}$ and let 
$$P_{R_{[{\bf b- \varepsilon_i}]}}(t)=\frac{h_0({[{\bf b- \varepsilon_i}]})+h_1({[{\bf b- \varepsilon_i}]})t+...+h_r({[{\bf b- \varepsilon_i}]})t^r }{(1-t)^{d-1}},$$ with
 $h_r({[{\bf b- \varepsilon_i}]})\not=0,$
 be the Hilbert-Poincaré series of 
$\Sigma _{[{\bf b- \varepsilon_i}]}$. Then we have 
$$P_{R_{[{\bf b}]}}(t)=\frac{{h}_0({[{\bf b}]})+{h}_1({[{\bf b}]})t+...+{h}_r({[{\bf b}]})t^r+ {h}_{r+1}({[{\bf b}]})t^{r+1} }{(1-t)^{d}},$$
where: $${h}_0({[{\bf b}]})=1, \ \  \mbox{for \ }k=1,...,r;\ \  h_k({[{\bf b}]})=\frac{ (d-1-b_i-(k-1))h_{k-1}({[{\bf b- \varepsilon_i}]})+(b_i+k)h_{k}({[{\bf b- \varepsilon_i}]}) }{b_i},$$ 
and $\displaystyle h_{r+1}({[{\bf b}]})=\frac{ (d-1-b_i-r)h_{r}({[{\bf b- \varepsilon_i}]}) }{b_i}$. Note that $\displaystyle h_{r+1}({[{\bf b}]})$  can be zero.
 If $\displaystyle h_{r+1}({[{\bf b}]})=0$ then $\displaystyle h_r({[{\bf b}]})=\frac{ h_{r-1}({[{\bf b- \varepsilon_i}]})+(b_i+r)h_{r}({[{\bf b- \varepsilon_i}]}) }{b_i}>0.$
\end{theorem}
\begin{proof}
Let $b_i\geq 1$. Let remark that ${{b_i+l}\choose{b_i}}= {{b_i-1+l}\choose{b_i-1}}\times \frac {b_i+l} {b_i},$ so that 
$$H_{R_{[{\bf b}]}}(l)  = (H_{R_{[{\bf b- \varepsilon_i}]}}(l)) \times \frac {b_i+l} {b_i} ,$$
which implies that:
$$ P_{R_{[{\bf b}]}}(t)=\sum_{l\geq 0} (H_{R{[{\bf b- \varepsilon_i}]}}(l)) \times \frac {b_i+l} {b_i}t^l $$
$$ P_{R_{[{\bf b}]}}(t)=\sum_{l\geq 0} H_{R{[{\bf b- \varepsilon_i}]}}(l) t^l+
\frac {1} {b_i}\sum_{l\geq 0} l H_{R_{[{\bf b- \varepsilon_i}]}}(l) t^l, $$ so that

$$P_{R_{[{\bf b}]}}(t)=P_{R_{[{\bf b-\varepsilon _i}]}}(t)+ \frac {1} {b_i} t P'_{R_{[{\bf b-\varepsilon _i}]}}(t),$$ where 
$P'_{R_{[{\bf b-\varepsilon _i}]}}(t)$ is the derivative of $P_{R_{[{\bf b-\varepsilon _i}]}}(t).$\hfill\break 
Since $P_{R_{[{\bf b- \varepsilon_i}]}}(t)=\frac{h_0({[{\bf b- \varepsilon_i}]})+h_1({[{\bf b- \varepsilon_i}]})t+...+h_r({[{\bf b- \varepsilon_i}]})t^r }{(1-t)^{d-1}},$ 
 with $d=b_1+b_2+...+b_n+1$, we get :
$$P_{R_{[{\bf b}]}}(t)=  \frac{\sum_{k=0}^{r+1}\frac {1} {b_i}( (d-1-b_i-(k-1))h_{k-1}({[{\bf b- \varepsilon_i}]})+(b_i+k)h_{k}({[{\bf b- \varepsilon_i}]})  )t^k  }{(1-t)^{d}} ,$$
 the proof is over.
\end{proof}
\begin{remark} As we will see in Section \ref{Sorted}, the numbers $h_k({[{\bf b}]})$ coincide with the Simon Newcomb's numbers, denoted by historical reasons $A([{\bf b}],k)$.
\end{remark}
We summarize some known results on the Simon Newcomb numbers, proved  by probabilistic methods.  
\begin{proposition}\begin{enumerate}
\item (Dillon and Roselle \cite{dr}(1969))$\mbox {For  \ }k=0 ,...,r :   $
\begin{equation}\label{h-vector-hilbertfunction2bis} A([{\bf b}],k)= \sum_{j=0}^k (-1)^j {{d}\choose{j}} 
{{b_1+k-j}\choose{b_1}}\times {{b_2+k-j}\choose{b_2}}\times ... \times {{b_n+k-j}\choose{b_n}}   .\end{equation}
\item  \cite{flw}  We have  the following recursive formula. $\mbox {For  \ }k\geq 1  $:
$$A([{\bf b}],k-1)=$$ $$=\frac{(N-b_n-k+1)A([{\bf b-\varepsilon _n}],k-2)+(k+b_n) A([{\bf b-\varepsilon _n}],k-1)}{b_n}$$
where 
$N=b_1+...+b_n$, 
\end{enumerate}

\end{proposition}
 The proof of 1.  follows immediately from the Equality \ref{h-vector-hilbertfunction2}. Claim  2. is the statement of  Theorem \ref{poincarerecurrence}.
 We pointed that in  \cite{flw} we can find the following recursive formula:
$$A([{\bf b}],k)=(N-l(N)-k+1)A([{\bf b-\varepsilon _n}],k-1)+(k+l(N)) A([{\bf b-\varepsilon _n}],k)$$
where 
$N=b_1+...+b_n$ and $l(N)=b_n$, and the  notation $A([{\bf b}],k)$ in \cite{flw} is in fact 
$A([{\bf b}],k-1)$ in our notations.
 As we can see this formula is wrong, we have found the right formula by reading carefully the proof in  \cite{flw}, this mistake motivates us to find another direct proof.
\begin{corollary}  Let $(h_0({[{\bf b}]}), h_1({[{\bf b}]}),..., h_{r_{[{\bf b}]}}({[{\bf b}]}))$ be the $h-$vector of the Hilbert-Poincaré series of 
$R_{[{\bf b}]}$, with $h_{r_{[{\bf b}]}}({[{\bf b}]})\not=0,$ and $b^*= \max\{ b_1,b_2,...,b_n\}$. Then we have 
\begin{enumerate}
\item $r_{{[{\bf b}]}}=(b_1+b_2+...+b_n)-b^*. $
\item  $h_{r_{[{\bf b}]}}({[{\bf b}]})={{b^*}\choose{b_1}}...{{b^*}\choose{b_n}}.$ In particular 
$h_{r_{[{\bf b}]}}({[{\bf b}]})=1$ if and only if $b_1=b_2=...=b_n$, in this case $r_{{[{\bf b}]}}=(n-1)b_1$. That means that
 $R_{[{\bf b}]}$
 is arithmetically Gorenstein if and only $b_1=b_2=...=b_n$. This was proved in  \cite{gw}.
\end{enumerate} 
\end{corollary}

\begin{proof}    Dillon and Roselle \cite{dr} prove this statement. 
Our purpose is to prove the Corollary only by elementary calculations on the Hilbert-Poincaré series. The proof will be by induction on the sum 
 $d :=b_1+b_2+...+b_n+1$. The case $d=2$ occurs if and only if $n=1, b_1=1$, in this case we have that the $h-$vector is $(1)$, so the claim is true.
Assume that the claim is proved for a natural  number $d\geq 2$. Let $b_1,b_2,...,b_n$ a sequence of non zero natural numbers such that $d=b_1+b_2+...+b_n+1. $ 
Without loss of generality we can assume that $b_1\geq b_2\geq ...\geq b_n$.
By induction hypothesis we have that $r_{{[{\bf b- \varepsilon_n}]}}=d-2-b_1,$ and $h_{r_{[{\bf b- \varepsilon_n}]}}({[{\bf b- \varepsilon_n}]})={{b_1}\choose{b_1}}...{{b_1}\choose{b_n-1}}.$ 
From the Theorem \ref{poincarerecurrence} we have that $r_{{[{\bf b}]}}\leq r_{{[{\bf b- \varepsilon_n}]}}+1=d-1-b_1, $ and 
\begin{align} h_{d-1-b_1}({[{\bf b}]})& =\frac{ (d-1-b_n-r_{{[{\bf b- \varepsilon_n}]}})h_{r_{{[{\bf b- \varepsilon_n}]}}}({[{\bf b- \varepsilon_i}]}) }{b_n}\\ 
& = \frac{ (b_1-(b_n-1))h_{r_{{[{\bf b- \varepsilon_n}]}}}({[{\bf b- \varepsilon_i}]}) }{b_n}\\ 
& = \frac{ (b_1-(b_n-1)){{b_1}\choose{b_1}}...{{b_1}\choose{b_n-1}}) }{b_n}\\ 
& ={{b_1}\choose{b_1}}...{{b_1}\choose{b_n}}.
\end{align}
The proof is complete.
\end{proof}
\begin{remark}It is a general fact that $h_0{{[{\bf b}]}}=1, h_1{{[{\bf b}]}}=(b_1+1)\times ...\times (b_n+1)-(b_1+b_2+...,+b_n)-1$ and
 $h_1{{[{\bf b}]}}+h_2{{[{\bf b}]}}+...,+h_{r_{{[{\bf b}]}}}({{[{\bf b}]}})= \deg \Sigma _{{{[{\bf b}]}}}=\frac{(b_1+b_2+...+b_n)!}{b_1!b_2!...b_n!}.$
\end{remark}
 \begin{corollary} For $n\geq 3$, we can compute the $h-$polynomial in the cases where the regularity $r_{{{[{\bf b}]}}}$ is 2 or 3.
\begin{enumerate}
\item  The case $r_{{{[{\bf b}]}}}=2$, can arrive if and only if $n=3, b_1=1,b_2=1$ and $b:=b_3$ any.
the $h-$polynomial is $1+(3b+1)t+ b^2t^2.$
\item The case $r_{{{[{\bf b}]}}}=3$, can arrive if and only if $n=3, b_1=1,b_2=2$ and $b:=b_3\geq 2$, or $n=4, b_1=b_2=b_3=1$ and $b:=b_4\geq 1$.
\begin{itemize}
\item if $n=3, b_1=1,b_2=2$ and $b:=b_3\geq 2$, then $h_0=1, h_1=5b+2, h_2=\displaystyle\frac{7b^2+b}{2},h_3=\displaystyle\frac{b^2(b-1)}{2}$.
\item if $n=4, b_1=b_2=b_3=1$, and $b:=b_4\geq 1$, we have $h_0=1, h_1=7b+4, h_2= 6b^2+4b+1, h_3= b^3. $
 \end{itemize}
 \end{enumerate}
\end{corollary}

The proof follows immediately from the above remark. 
\subsection{Sorted sets,  Grobner basis} \label{Sorted}

We  order the elements in $\bn^n$ by saying that $(v_1,...,v_n)\leq (w_1,...,w_n) $ if and only if $v_i\leq w_i$ for all $i=1,...,n$.
 Let $b_1,...,b_n\in \bn^* $ and $\Mcal= \{0,...,b_1\}\times \{0,...,b_2\}\times...\times  \{0,...,b_n\}\subset \bn^n$ with the induced order. For any $u,v\in \Mcal$ we define $U(v,w),V(v,w)$ by setting
$$ U(u,v)_i=\min \{u_i,v_i\}, V(u,v)_i=\max \{u_i,v_i\}.$$ $U(v,w),V(v,w)$ are also the unique elements in $\Mcal$ such that $U(v,w)\leq v,w, V(v,w)\geq v,w$ and $U(v,w)$ (resp. $V(v,w))$) is maximal, (resp. minimal) with respect this property.

  Let $S:=k[T_v |v\in \Mcal]$, for any vectors  $v_1,v_2,...,v_s\in \Mcal$, we will say that the monomial $T_{v_1}T_{v_2 }...T_{v_s }$ is sorted if $v_1\leq v_2\leq ...\leq v_s$.
 An unsorted monomial has at least two no comparable vectors, that is   any unsorted monomial has a quadratic factor unsorted monomial. For any unsorted monomial $T_vT_w$, there is a unique sorted monomial $T_{U(v,w)}T_{V(v,w)}$.

\begin{theorem}\label{grobnerbasis}
 Let $\Mcal= \{0,...,b_1\}\times \{0,...,b_2\}\times...\times  \{0,...,b_n\}\subset \bn^n$. Let $\Jcal_{[{\bf b}]}$ be   the kernel of the ring homomorphism 
$$f : k[T_v |v\in \Mcal] \longrightarrow  k[x_{1,0}, ...,x_{1,b_1},x_{2,0}, ...,x_{2,b_2},...,x_{n,0}, ...,x_{n,b_n}]$$ given by 
$$T_v\mapsto x_{1,v_1}x_{2,v_2}...x_{n,v_n}.$$ We will denote by $\Gcal_\Mcal$ the set of all  the nonzero  binomials
$${\underline {T_vT_w}}-T_{U(v,w)}T_{V(v,w)}.$$ Let denote by  $\prec_M$ any revlex order on $k[T_v |v\in \Mcal]$ compatible with the diagonal order, then $\Gcal_\Mcal$ 
is a reduced Groebner basis of $\Jcal_{[{\bf b}]}$, 
the underlined monomials being the leading terms.
\end{theorem}

\begin{lemma}Any monomial can be sorted modulo $\Gcal_\Mcal$ after a finite number of steps. A binomial in $\Jcal_{[{\bf b}]}$ which is the difference of two sorted monomials is zero.
\end{lemma}
 \begin{proof}
By sorting a quadratic monomial we mean to replace it modulo $\Gcal_\Mcal$, that is $T_{v}T_{w}$ is sorted by
$T_{U({v},{w})}T_{V({v},{w})}$. We prove the lemma by induction on the degree of the monomial. Let $T_{v_1}T_{v_2 }...T_{v_s }$ be a monomial.
 It is clear that our assertion is true for $s=1$ or $s=2$. Let $s\geq 3$ and we suppose that $T_{v_1}T_{v_2 }...T_{v_s }$ is unsorted. 
If there exists some vector $v_1$ (after possible re-indexing)such that  ${v_1}\leq  {v_j }, \forall j\geq 2$, 
then by induction hypothesis the monomial $T_{v_1}T_{v_2 }...T_{v_s }$ can be sorted. So we suppose that   $T_{v_1}T_{v_{2}}$ is unsorted, 
set $v^{(1)}=U({v_1},{v_{2}}), v'_{2}=V({v_1},{v_{2}})$.  Set $v^{(k)}=U({v^{(k-1)}},{v_{k+1}}), v'_{k}=V({v^{(k-2)}},{v_{k}})$, for $k\geq 3$.
 The monomial  $T_{v^{(s-1)}}T_{v'_2 }...T_{v'_s }$ is may be unsorted but has the property that 
$ v^{(s-1)}\leq  {v'_j }, \forall j\geq 2$ and so by the above argument  $T_{v_1}T_{v_2 }...T_{v_s }$ is equivalent modulo $\Gcal_\Fcal$ to a  sorted monomial. 

Now let a binomial $ m_1-m_2:=T_{v_1}T_{v_2 }...T_{v_s }-T_{w_1}T_{w_2 }...T_{w_s } \in \Jcal_{[{\bf b}]}$, we prove that if $m_1, m_2$ are sorted then $m_1=m_2$.
 Suppose that there exist some $1\leq i\leq n$ such that $(v_{1})_i<(w_{1})_i$, then we will have that for any $k\geq 1$ $(v_{1})_i<(w_{k})_i$,
 this is a contradiction since $ m_1-m_2 \in \Jcal_{[{\bf b}]}$.
\end{proof}
 \begin{proof}{\bf Proof of Theorem \ref{grobnerbasis}}
 Now we consider any term order $\prec_\Mcal$ on $k[T_v |v\in \Mcal] $ compatible with the diagonal order and we take the reverse lexicographic order. 
It is clear that elements in $\Gcal_\Mcal$ are ordered by decreasing order. 
We prove now that $\Gcal_\Mcal$ is a Groebner basis of $\Jcal_{[{\bf b}]}$. By definition $\Gcal_\Mcal\subset \Jcal_{[{\bf b}]},$ 
 we take any monomial $m_1$ in $in_{\prec_\Mcal}(\Jcal_{[{\bf b}]})$, if $m_1$ is unsorted then it is a multiple of the initial term of some element in  $\Gcal_\Mcal$.
 If $m_1$ is sorted then there exists another monomial $m_2$, with $ m_1-m_2 \in \Jcal_{[{\bf b}]}$,  and $m_2$ is not in $in_{\prec_\Mcal}(\Jcal_{[{\bf b}]})$, 
  that is, $m_2$ is sorted, this implies that $m_1=m_2$ and proves our assertion.
\end{proof}
The following theorem is a consequence of Theorem \ref{grobnerbasis},  \cite{st}[chapter 8] and the primary decomposition of square free monomial ideals.
 \begin{theorem}\label{simplicial}
The toric ring $R_{[{\bf b}]}:=k[T_v |v\in \Mcal]/\Jcal_{[{\bf b}]} $ is a Cohen-Macaulay ring. We consider the Stanley Reisner ring associated to $in_{\prec_\Mcal}(\Jcal_{[{\bf b}]})$, 
and its simplicial complex $\Delta_{\prec_\Mcal}$.  A face of  $\Delta_{\prec_\Mcal}$ is the support of any sorted monomial.
 Every maximal face of  $\Delta_{\prec_\Mcal}$ is given by a maximal sorted subset of $\Mcal$  and so has cardinal  $b_1+...+b_n+1$.
 The primary decomposition of  $in_{\prec_\Mcal}(\Jcal_{[{\bf b}]})$ is an intersection of linear ideals of height
 $(b_1+1)\times ...\times (b_n+1)-(b_1+...+b_n+1)$. In particular $\dim R_{[{\bf b}]}=b_1+...+b_n+1$. 
   The Segre polyhedron $\Pcal_{[{\bf b}]}:= Convex(\Mcal)\subset \br^n$ has a unimodular triangulation. The normalized volume of $\Pcal_{[{\bf b}]}$
is equal to the number of sorted subsets of $\Mcal$ of cardinal  $b_1+...+b_n+1$. So we have:
$$Vol (\Pcal_{[{\bf b}]})={{b_1+...+b_n}\choose{b_1}}\times {{b_2+...+b_n}\choose{b_2}}\times ...\times 
{{b_{n-1}+b_n}\choose{b_{n-1}}}.$$
Note that $${{b_1+...+b_n}\choose{b_1}}\times {{b_2+...+b_n}\choose{b_2}}\times ...\times 
{{b_{n-1}+b_n}\choose{b_{n-1}}}= \frac{(b_1+...+b_n)!}{b_1!... b_n!}.$$
\end{theorem}
Theorem \ref{simplicial} says that any maximal face of $\Delta_{\prec_\Mcal}$ is given by a maximal sorted subset of $\Mcal$,  and  has cardinal  $b_1+...+b_n+1$. 
On the other hand any  maximal sorted subset  of $\Mcal$, can be viewed as a chain of points   of $\Mcal$ of length  $b_1+...+b_n+1$, or a oriented walk in    $\Mcal$,
 with starting point the point $D=(0,...,0)$ and endpoint,
 or arrival, the point $A=(b_1,...,b_n)$. 
Next let $\Scal$ be the set of the sequences $(u_l)_{1\leq l\leq b_1+ b_2+ ...+ b_n}$ of length $b_1+ b_2+ ...+ b_n$ with $b_i$ times the symbol $i$. 
Let $\varepsilon_1,...,\varepsilon_n$ be the canonical basis in $\bz^n$, so any maximal sorted subset  of $\Mcal$ can be identified with a sequence $(u_l)\in \Scal$, in the following way:

To any sequence $(u_l)\in \Scal$ we associate the set of points in $\Mcal$, $$\Fcal_{(u_l)}=\{D+\sum_{i=1}^m\sum_{j=1}^n \delta_{u_i,j} \varepsilon_j\mid 1\leq m\leq b_1+ b_2+ ...+ b_n\}.$$
By the above identification $\Scal$  is the set of all maximal faces of $\Delta_{\prec_\Mcal}$.

We say that we have a descent in the sequence $(u_l)$, if $u_i>u_{i+1}$ for some $i.$ Let 
$$\Gcal_{(u_l)}=\{D+\sum_{i=1}^m(\sum_{j=1}^n \delta_{u_i,j} \varepsilon_j)\mid 1\leq m\leq b_1+ b_2+ ...+ b_n, u_m>u_{m+1}  \},$$
$\Gcal_{(u_l)}$ is called the descent set of $\Fcal_{(u_l)}$ (or of $(u_l)$).
We say that the sequence $(u_l)$ pass through a set $G$ if $G\subset \Fcal_{(u_l)}$.

\begin{theorem}\label{partition}We have a partitioning of the   initial complex 
$$\Delta_{\prec_\Mcal}(\Jcal_{[{\bf b}]})=\bigcup_{(u_l)\in \Scal} [\Gcal_{(u_l)}, \Fcal_{(u_l)}] .$$ 
and the Hilbert-Poincar\'e
 series is given by:
 $$P_{R_{[{\bf b}]}}(t)=\frac{A_{[{\bf b}]}(t)}{(1-t)^{b_1+ b_2+ ...+ b_n+1}}.$$
where $$A_{[{\bf b}]}(t):= \sum_{(u_l)\in \Scal}t^{{\rm desc}(u_l)}=\sum_{k} A({[{\bf b}]},k) t^{k}$$ is also known as the multiset Eulerian polynomial,and 
$A({[{\bf b}]},k) $ counts the number of sequences in $\Scal$ with exactly $k$ descents.  

\end{theorem}
\begin{proof}
The second assertion follows  from the first claim, indeed let $$\displaystyle A({[{\bf b}]},k)=\card \{ (u_l)\in \Scal \ |\ \card(\Gcal_{(u_l)})=k \}$$ then by application of 
\cite{st}[Proposition III.2.3], we have 
$$P_{R_{[{\bf b}]}}(t)=\frac{  \sum_{k} A({[{\bf b}]},k) t^{k}  }{(1-t)^{b_1+ b_2+ ...+ b_n+1}}.$$
The first assertion will follows from the next Lemma, where we define a surjective map 
$$h:\Delta_{\prec_\Mcal}\rightarrow \Scal,$$ so that 
$$\cup_{(u_k)\in\Scal} h^{-1}(u)$$ is the partition in the theorem.
\end{proof}
\begin{lemma} Let $G$ be any face of $\Delta_{\prec_\Mcal}$, that is a set of increasing points $P_1\leq P_1\leq  ... \leq P_s\in \Mcal$. 
There exist a well defined sequence $(u_{G,k})$ passing through $G$,  such that all the descents points in $(u_{G,k})$  are inside  $G$.
\end{lemma}
We define $(u_{G,k})$ by induction on the cardinal of $G$.
If $G=\emptyset $ then we set $(u_{\emptyset ,k})$ the unique increasing sequence in $\Scal$.

Let $P=(v_1,...,v_n)\in \Mcal $ be any point, we associate to $P$ a sequence $(u_{P,k})\in \Scal$, defined by 
$$u_{P,k}=\begin{cases}1& \mbox{if  \ \  } 0\leq k\leq v_1,\cr
2& \mbox{if  \ \  }v_1< k\leq v_1+v_2,\cr
&....\cr
n& \mbox{if  \ \  } v_1+...+v_{n-1}< k\leq v_1+...+v_{n}=\mid P-D\mid,   \cr
1& \mbox{if  \ \  } v_1+...+v_{n}< k\leq b_1+v_2+...+v_{n},  \cr
2& \mbox{if  \ \  } b_1+v_2+...+v_{n}< k\leq b_1+b_2+v_3+...+v_{n},  \cr
&....\cr
n& \mbox{if  \ \  } b_1+...+b_{n-1}+v_{n}< k\leq b_1+b_2+...+b_{n}.  \cr
\end{cases}$$
Note that $(u_{\emptyset ,k})=(u_{D,k})=(u_{A,k}),$
 $(u_{P,k})$ has at most an descent point in $P$ and has not descent points if and only if 
$(u_{P,k})=(u_{\emptyset ,k}).$

For any two points $P=(u_1,...,u_n)\leq Q=(v_1,...,v_n)$ we associate the sequence $(u_{P,Q,k})\in \Scal$  defined by 
$$u_{P,Q,k}=\begin{cases} u_{P,k}& \mbox{if  \ \  } 0\leq k\leq \mid P-D\mid, \cr
1& \mbox{if  \ \  } \mid P-D\mid < k\leq \mid P-D\mid + (v_1-u_1),\cr
2& \mbox{if  \ \  } \mid P-D\mid + (v_1-u_1)< k\leq \mid P-D\mid + (v_1-u_1)+ (v_2-u_2),\cr
&....\cr
n& \mbox{if  \ \  } \mid P-D\mid + (v_1-u_1)+...+ (v_{n-1}-u_{n-1})< k\leq \mid Q-D\mid, \cr
u_{Q,k}& \mbox{if  \ \  }  \mid Q-D\mid< k\leq b_1+b_2+...+b_{n},\cr\end{cases}$$
where we have set $\mid P\mid=u_1+...+u_{n}$.

Let remark that $(u_{P,Q,k})$ has at most two descent points in $P$ or $Q$.

For any set of increasing points $P_1\leq P_2\leq ...\leq P_s$,  $(u_{P_1, P_2, ..., P_s,k})\in \Scal$  is defined by 
induction:
$$u_{P_1, P_2, ..., P_s,k}=\begin{cases}
u_{P_1,P_2,k}& \mbox{if  \ \  } 0\leq k\leq   \mid P_2-D\mid, \cr
u_{P_2, ..., P_s,k}& \mbox{if  \ \  }  \mid P_2-D\mid< k\leq b_1+b_2+...+b_{n}.\cr\end{cases}$$
It follows from the construction that $(u_{P_1, P_2, ..., P_s,k})$ pass through  $P_1, P_2, ..., P_s$ and all descent points of  $(u_{P_1, P_2, ..., P_s,k})$ are among $P_1, P_2, ..., P_s$.

Now we prove by induction on $s$ that if $\Gcal$ is the descent set of $(u_{P_1, P_2, ..., P_s,k})$ then $(u_{\Gcal,k})$=$(u_{P_1, P_2, ..., P_s,k})$. The case $s=0$ is clear, 
 so let $s\geq 2$, and  $\Gcal'$ 
be the descent set of $(u_{ P_2, ..., P_s,k})$ then $(u_{\Gcal',k})=(u_{ P_2, ..., P_s,k})$,
we have to consider several cases:
\begin{enumerate}
\item $P_2\not\in \Gcal'$, this implies that $P_2\not\in \Gcal$, so have two subcases:
\begin{enumerate}
\item $P_1 \in  (u_{P_2,k}),$ in this case $\Gcal= \Gcal', (u_{ P_1, ..., P_s,k})=(u_{\Gcal',k})$
\item $P_1\not\in (u_{P_2,k}),$ in this case $\Gcal=\{P_1\}\cup  \Gcal', (u_{ P_1, ..., P_s,k})=(u_{\Gcal,k})$
\end{enumerate}

\item $P_2\in \Gcal'$,  have three subcases:
\begin{enumerate}
\item $P_1\in (u_{P_2,k}),$ in this case $\Gcal= \Gcal', (u_{ P_1, ..., P_s,k})=(u_{\Gcal',k})$
\item $P_1\not\in (u_{P_2,k}),$ $p_2\not\in \Gcal,$ then  $\Gcal=\{P_1\}\cup  (\Gcal'\setminus \{P_2\}) ,$ and $ (u_{ P_1, ..., P_s,k})=(u_{\Gcal,k})$
\item $P_1\not\in (u_{P_2,k}),$ $p_2\in \Gcal,$ then  $\Gcal=\{P_1\}\cup  (\Gcal') ,$ and $ (u_{ P_1, ..., P_s,k})=(u_{\Gcal,k})$.
\end{enumerate}
\end{enumerate}
The following corollary gives the Simon Newcomb's numbers as a sum of positive integers, so solves the original question of Simon Newcomb.
\begin{corollary}{\bf  Simon Newcomb's Conjecture.} For $k =0,..., b_{1}+...+b_{n}-\max \{b_{1},...,b_{n}\}$, we have 
$$A({[{\bf b}]},k)=\sum_{(i_2,...,i_{n-1})\in \Delta } A_{i_2}  A_{i_2,i_3}A_{i_3,i_4}... A_{i_{n-1},i_{n}},$$
where $$i_{n}:=k; A_{i_2}={ {b_1 }\choose{i_2}}{ {b_2}\choose {i_2}}; \  \forall  s= 2,...,n-1,\  A_{i_s,i_{s+1}}={ {b_1+...+b_{s}-i_{s}}\choose{i_{s+1}-i_{s}}}{ {b_{s+1}+i_{s}}\choose { i_{s+1}}}\geq 0$$ 
 and $\Delta $ is defined by 
\begin{eqnarray*}0\leq   &i_{2}&\leq \min \{ b_{1}+b_{2}-\max \{b_{1},b_{2}\}, i_{3}  \},\\
 0\leq   &i_{3}&\leq \min \{ b_{1}+b_{2}+b_3-\max \{b_{1},b_{2},b_3\}, i_{4}  \}, \\ ...,\\
 0\leq   &i_{n-1}&\leq \min \{ b_{1}+...+b_{n-1}-\max \{b_{1},...,b_{n-1}\}, i_{n}  \}\end{eqnarray*}

\end{corollary}
The proof follows  from the theorem \ref{partition} and by induction on the next theorem (see \cite{fk},\cite{md}).
Consider  formal Laurent series
$$ \mathfrak a=\sum_{l\geqslant \sigma_a}a_lt^l, \sigma_\mathfrak a\in \fz, a_l\in \fc $$
such that 
$$ (*)\ \ \  \mathfrak a=\frac{h(\mathfrak a)(t)}{(1-t)^{d_\mathfrak a}}, \text{ \ for some } d_\mathfrak a\geqslant 0, h(\mathfrak a)(t)\in \fc[t,t^{-1}]. $$

\begin{theorem}\label{t2} Let $\mathfrak a, \mathfrak b $ be any formal power Laurent series satisfying property (*).
 Let $b_{\mathfrak a}=d_{\mathfrak a}-1\geqslant 0, b_{\mathfrak b}=d_{\mathfrak b}-1\geqslant 0$ and $\sigma $ be any 
of the numbers $\sigma_{\mathfrak a}, \sigma_{\mathfrak b}, \sigma_{(\mathfrak a\otimes \mathfrak b)}=\max(\sigma_{\mathfrak a}, \sigma_{\mathfrak b}),\min(\sigma_{\mathfrak a}, \sigma_{\mathfrak b}). $
 Then for any $n\in \fz$ 
$$ h_n(\mathfrak a\otimes \mathfrak b)=\sum_{i=\sigma}^{\infty}\sum_{j=\sigma}^{\infty}h_i(\mathfrak a)h_j(\mathfrak b){ {b_{\mathfrak a}+j-i}\choose{n-i}}{ {b_{\mathfrak b}+i-j}\choose {n-j}}. $$ 
\end{theorem}
\begin{remark} If $b_1=1,b_2=2,b_3=2$; 
the sequence of points $(0,0,0),(0,1,0),(0,2,0),(0,2,1),$ $ (0,2,2),(1,2,2)$ corresponds to the sequence $ 2 2 3 3 1$ and has one descent,
 so the situation in the 3-dimensional space is different from the situation in the plane.

 The sequence of points $(0,0,0),(0,0,1),(0,1,1),(0,2,1), (0,2,2),(1,2,2)$  corresponds to the sequence $3 2 2  3 1 $ has two descents
\end{remark}
\begin{remark} It follows from the above theorem and the fact that the Segre embedding is arithmetically Cohen-Macaulay, that $A({[{\bf b}]},0)=1$, $A({[{\bf b}]},1)$ is the codimension
of the Segre embedding, and the Castelnuovo-Mumford regularity $r_{[{\bf b}]}$ of $R_{[{\bf b}]}$ is the degree of 
$A_{[{\bf b}]}(t)$, that is the highest $k$ such that  $A({[{\bf b}]},k) \not=0.$
\end{remark}

Our Theorem \ref{partition} is important even in  some particular cases:
\begin{example}
The Gorenstein case. The ring $R_{[{\bf b}]}$ is Gorenstein if and only if $h_{r_{[{\bf b}]}}=1$, that is, if and only if $b_1=...=b_n$. 
This was proved in   \cite{gw}. In this case  the Castelnuovo-Mumford regularity of $R_{[{\bf b}]}$ is $b_1(n-1)$. The case $b_1=1$ is well known, 
$A([1,...,1],k)$ is the Eulerian number $A(n,k)$ which counts the number of permutation of $n$ elements  with $k$ descents. 
We have that the  Ehrhard function $A_{\Acal}(r)$   coincides with the Hilbert polynomial. Let $H(t)=\sum_0^{+\infty }A_{\Acal}(r)t^r$ be the Hilbert-Poincar\'e series. 
We have the following identity:
$$H(t)=\sum_{k=0}^{+\infty } (k+1)^n t^k=\displaystyle\frac{\sum_{k=0}^{n-1 } A(n,k)t^k}{(1-t)^{n+1}}.$$
It is also well known that
$$A(n,k) =kA(n-1,k)+ n A(n,k-1)$$
\end{example}

\begin{example}
Moreover in the case $b_1=...=b_n=1$, Theorem \ref{partition} give us the familiar unimodular triangularization of the unit hypercube in dimension $n$ into $n!$ simplices of volume one. 
For example in 
the case $[1,1,1]$,
 we are working on the ring $S=k[T_{i,j,k} \mid 0\leq i,j,k\leq 1]$. For the above term order the triangulation of the polyhedron 
$\Pcal_{[{\bf b}]}$, the initial complex $\Delta_{\prec_\Mcal}(\Jcal_{[{\bf b}]}) $ can
 be represented by the following hexagon (by omitting the point $(1,1,1))$, which is in all the facets)
\begin{center}
\includegraphics[height=3 in,width=3in]{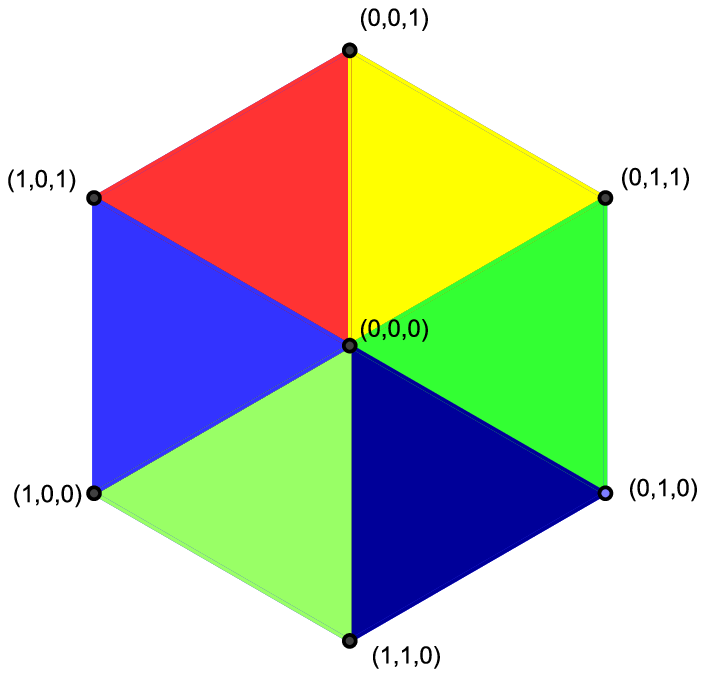}
\end{center}
Note that $in_{\prec_\Mcal}(\Jcal_{[{\bf b}]}) $ is generated by all the 9 diagonals in this picture.
The $h-$polynomial of $R_{[{\bf b}]}$ is $h(t)=1+4t+t^2$. Remark that any cyclic order on the edges gives the decomposition of 
$\Delta_{\prec_\Mcal}(\Jcal_{[{\bf b}]}) $ as a shellable complex.
\end{example}
\begin{example}
The case $[1,1,2]$.
In this case we are working on the ring $S=k[T_{i,j,k} \mid 0\leq i,j\leq 1, 0\leq k\leq 2]$. For the above term order 
the initial complex $\Delta_{\prec_\Mcal}(\Jcal_{[{\bf b}]}) $, thanks to Theorem \ref{partition},  
can be represented (by omitting the points $(0,0,0),(1,1,2)$, which are in all the facets) by
 \begin{center}
\includegraphics[height=3 in,width=5in]{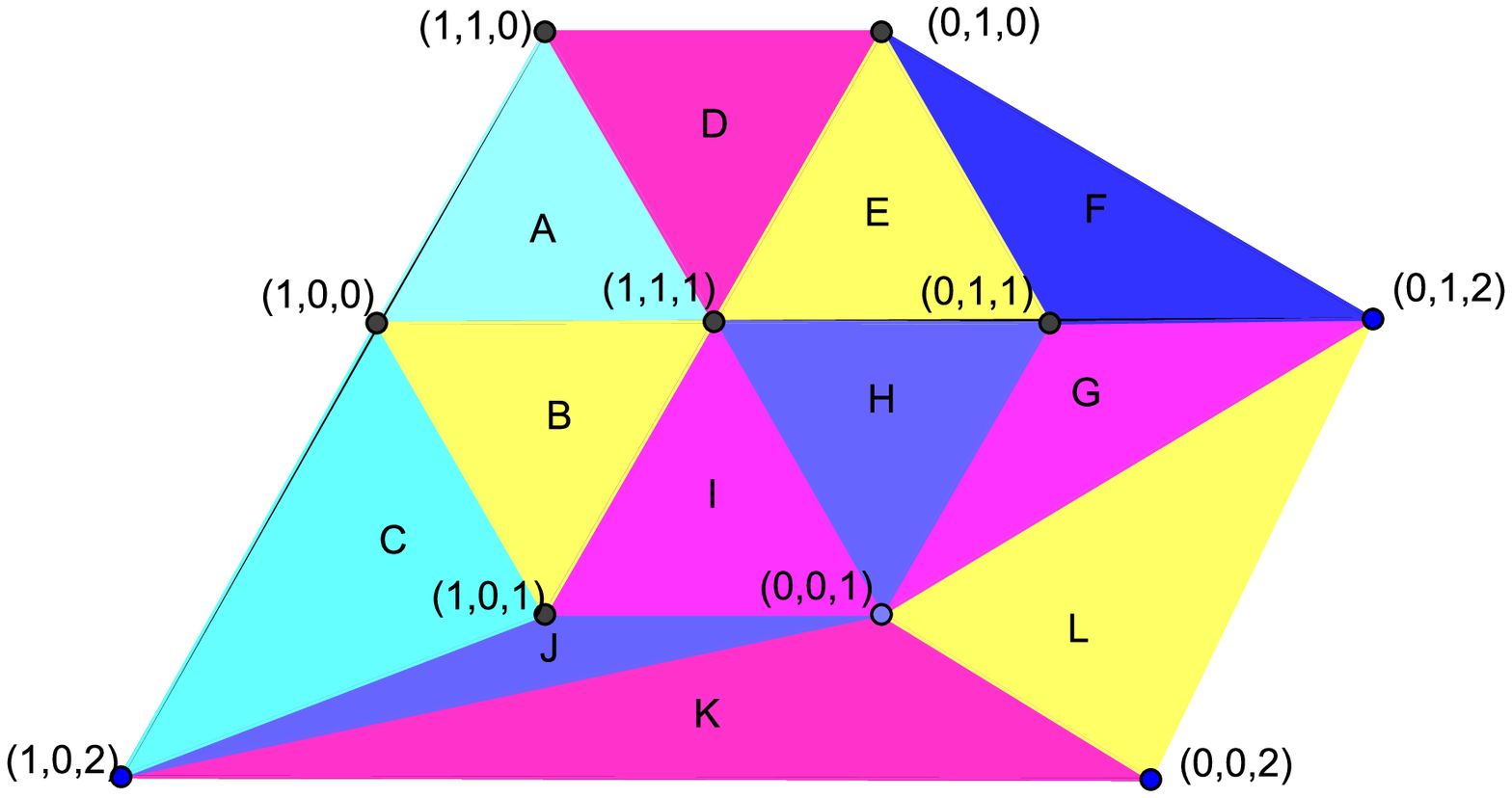}
\end{center}
The $h-$polynomial of $R_{[{\bf b}]}$ is $h(t)=1+7t+4t^2$. Remark that the  order on the facets ("triangles") by the lexicographic order gives the decomposition of 
$\Delta_{\prec_\Mcal}(\Jcal_{[{\bf b}]}) $ as a shellable complex. Note that $in_{\prec_\Mcal}(\Jcal_{[{\bf b}]}) $ is generated by all the 24 diagonals in this picture.

\end{example}
\subsection{Free Resolutions, Betti numbers}

Let $V$ be a set of variables and $S=K[V]$ the polynomial ring on the set of variables $V$ over a field $K$, graded by the standard graduation.
 Let $I\subset S$ a graded ideal of $S$. A graded minimal free resolution of $I$ is given by:
$$0\rightarrow \bigoplus_{\begin{array}{c}j=1 \end{array}}^{n_\rho} S(-a_{j,\rho})^{\beta_{\rho,a_{j,\rho}}}\rightarrow 
\dots\rightarrow \bigoplus_{\begin{array}{c}j=1 \end{array}}^{n_0} S(-a_{j,0})^{\beta_{0,a_{j,0}}}\rightarrow I \rightarrow 0,$$
where $a_{j,i}\in \Nbb$ and ${\beta_{i,a_{j,i}}}\in\Nbb^*$. The numbers ${\beta_{i,a_{j,i}}}$ are called graded 
 Betti numbers of $S/I$, sometimes to avoid any confusion we will denote by ${\beta_{i,a_{j,i}}}(I)$. Note that ${\beta_{i,a_{j,i}}}(I)= {\beta_{i+1,a_{j,i+1}}}(S/I).$
 
\begin{definition}\label{n2P-property}
Let $I\subset S$ be a homogeneous ($\Nbb$-graded standard) ideal of $S$ and $p\in\Nbb$, $p\neq 0$. 
We say that $I$ satisfies the ${\rm N}_{2,p}$ property if and only if for all $i\leq p-1$ and $j\geq 1$ $\beta_{i,i+2+j}(I)=0$.
\end{definition}
\begin{lemma} Let $p_2(I)$ the biggest integer such that $I$ satisfies the ${\rm N}_{2,p}$ property. Let 
$P_{S/I}(t)=\frac{1+h_1t+h_2t^2+h_3t^3+h_4t^4+...}{(1-t)^{d}}$ be the Hilbert-Poincaré series of $S/I$.
Then 
$${\rm for\ } i=1,...,p_2(I): \ \ \beta_{i,i+2}(I)= \sum_ {j=0}^{i+2} (-1)^{j+1}{{c }\choose{i+2-j}}h_j, $$

\end{lemma}
\begin{proof}\label{betti-h} It follows from the following identity:
$$\frac{1-\beta_{0,2}t^2+\beta_{1,3}t^3+...+(-1)^{p_2(I)+1}\beta_{p_2(I),p_2(I)+2}t^{p_2(I)+2}+ t^{p_2(I)+3}(...)}{(1-t)^{d+c}}$$  $$=
\frac{(1+h_1t+h_2t^2+h_3t^3+...+h_{p_2(I)+2}t^{p_2(I)+2}+ ...)\times (1-t)^{c}}{(1-t)^{d+c}}$$
\end{proof}
As before we set ${[{\bf b}]}={[b_1,...,b_n]}$. Now we apply Theorem \ref{partition} to give some results on the Betti numbers of Segre rings $R_{[{\bf b}]}$. 
We set $c({[{\bf b}]})= (b_1+1)\times ...\times (b_n+1)- (b_1+...+b_n+1) $ be the codimension of $R_{[{\bf b}]}.$ The following Theorem and its proof will be very useful
 for our results in this section.
\begin{theorem}\label{rubeiN2p}\cite{Ru2}[Theorem 10]Let $n\geq 3$ and $b_1,...,b_n\in\Nbb^*$ then the ideal $\Jcal_{[{\bf b}]}$ satisfy ${\rm N}_{2,p}$ property if and only $p\leq 3$.
\end{theorem}

The following corollary is an immediate consequence of Lemma \ref{betti-h}. Note that Simon Newcomb's numbers $A({[{\bf b}]}),j)$ can be computed easily for $j\leq 5$ by the equality
(\ref{h-vector-hilbertfunction2}).
\begin{corollary} 
We have 
$$\beta_{0,2}(\Jcal_{[{\bf b}]})={{(b_1+1)\times ...\times (b_n+1)-1+2 }\choose{2}}- {{b_1+2 }\choose{2}}\times ...\times {{b_n+2 }\choose{2}},$$
and  $\beta_{1,3}(\Jcal_{[{\bf b}]}),\beta_{2,4}(\Jcal_{[{\bf b}]}),\beta_{3,5}(\Jcal_{[{\bf b}]})$ are given in terms of the Simon Newcomb's numbers by using the formula:
$${\rm for\ } i=1,2,3: \ \ \beta_{i,i+2}(\Jcal_{[{\bf b}]})= \sum_ {j=0}^{i+2} (-1)^{j+1}{{c({[{\bf b}]}) }\choose{i+2-j}}A({[{\bf b}]},j), $$
\end{corollary} 

\begin{proposition}\label{ru}Let $m\geq n\geq 3$ and $K=\bc$. Let ${[{\bf b}]}={[b_1,...,b_n]}, {[{\bf b'}]}={[b'_1,...,b'_m]}$, assume that $ b'_1\geq b_1,..., b'_n\geq b_n.$ 
 Then  $\beta_{p,q}(R_{[{\bf b'}]})\geq \beta_{p,q}(R_{[{\bf b}]})$ for all $p,q$. In particular 
if $\beta_{p,q}(R_{[{\bf b}]})\not=0$, then $\beta_{p,q}(R_{[{\bf b'}]})\not=0,$ for any
 ${[b'_1,...,b'_m]}$ such that $m\geq n, b'_1\geq b_1,..., b'_n\geq b_n.$
\end{proposition}

\begin{proof}
 The proof follows by reading carefully \cite{Ru2}[Proof of Proposition 12(a)] and the remark of section 2 \cite{Gr}. 
For the commodity of the reader 
we give now a sketch of the proof:
It is well known that $\beta_{p,q}(R_{[{\bf b}]}) = \dim_{\bc} Tor^{S}_p(R_{[{\bf b}]},\bc )_{p+q}$. Now because $R_{[{\bf b}]}$ has
 a multigraduation $Tor^{S}_p(R_{[{\bf b}]},\bc )_{p+q}$ is a direct sum of Schur representations,
 that is for $i$ fixed we can write $Tor^{S}_p(R_{[{\bf b}]},\bc )_{p+q}=\oplus_{\lambda \in A^p_{V_i}} W_\lambda \otimes S^\lambda V_i$ with dim $V_i=b_i+1$, the important fact follows 
from the remark of section 2 \cite{Gr} that $A^p_{V_i}\subset A^p_{V'_i}$ if  $V_i\subset V'_i$, in particular if 
$Tor^{S}_p(S/\Jcal_{[b_1,...,b_{i-1},b_i,b_{i+1},...,b_n]},\bc )_{p+q}\not=0$ 
then $Tor^{S}_p(S/\Jcal_{[b_1,...,b_{i-1},b_i+1,b_{i+1},...,b_n]},\bc )_{p+q}\not=0$.
\end{proof}
\begin{corollary}\label{betti} Let $m\geq n\geq 3$, $K=\bc$. Let ${[{\bf b}]}={[b_1,...,b_n]}, {[{\bf b'}]}={[b'_1,...,b'_m]}$, assume that $ b'_1\geq b_1,..., b'_n\geq b_n.$ 
 We have 
 $$0<\beta_{c({[{\bf b}]}),c({[{\bf b}]})+\reg(R_{[{\bf b}]})}(R_{[{\bf b}]})\leq \beta_{c{[{\bf b}]},c({[{\bf b}]})+\reg(R_{[{\bf b}]})}(R_{[{\bf b'}]}).$$
As a special case we have that for any $ m\geq 3,$ 
 $$\beta_{2^m-m-2,2^m-2}(\Jcal_{[{\bf b}]})\geq \beta_{2^m-m-2,2^m-2}(\Jcal_{\underbrace{1,...,1}_{m }})\not=0.$$
The Betti diagram of the ideal $\Jcal_{[{\bf b}]}$ 
has the following shape:

\begin{tabular}{|c|c|c|c|c|c|c|c|c|c|}
  \hline
   & 0 & 1 & 2 & 3&4 &5 &... & ...&$c({[{\bf b}]})-1$\\
  \hline
 2& $\beta_{0,2}$& $\beta_{1,3}$ & $\beta_{2,4}$ &$\beta_{3,5}$ & & & & &... \\
  \hline
 3 &0 & 0 &0 &$\star $ & & & & &...\\
  \hline
  4 &0 & 0 &0 &... & & & & &...\\ \hline
  ... &0 & 0 &0 &...  & & & & &...\\
  \hline
 $b_1+...+b_n-b^*+1$ &0 & 0 &0 &.... & & & & &${{b^* }\choose{b_1}}\times ...\times {{b^*}\choose{b_n}}$\\
 \hline
\end{tabular}

where $b^*=\max\{b_1,...,b_n\} $ and $\star $ is non zero. 
\end{corollary}

\begin{proof} Since $R_{[{\bf b}]}$ is a Cohen-Macaulay ring we know that 
$$\beta_{c({[{\bf b}]}),c({[{\bf b}]})+\reg(R_{[{\bf b}]})}(R_{[{\bf b}]})=A({[{\bf b}]},\reg(R_{[{\bf b}]}))\not= 0,$$ so the first assertion follows from
 the  Proposition \ref {ru}.
In order to prove  that $\star $ is non zero, let write down the Betti diagram of  $\Jcal_{[1,1,1]}$ :

\begin{tabular}{|c|c|c|c|c|}
  \hline
   & 0 & 1 & 2 &3\\
  \hline
 2& 9& 16 & 9 &0  \\
  \hline
 3 &0 & 0 &0 &1 \\
  \hline 
\end{tabular}

Hence $\beta_{3,6}(\Jcal_{[1,1,1]})\not=0$, which by the  Proposition \ref {ru} implies $\beta_{3,6}(\Jcal_{[{\bf b}]})\not=0$ for any ${[{\bf b}]}$ with $n\geq 3$.
\end{proof}
\begin{example}\label{example-fin}In the cases where the regularity of $\Jcal_{[{\bf b}]}$ is 3 or 4, we can compute the above Betti numbers.
(We suppose $n\geq 3$). 

The Betti diagram of  $\Jcal_{[1,1,2]}$ :

\begin{tabular}{|c|c|c|c|c|c|c|c|}
  \hline
   & 0 & 1 & 2 &3&4&5&6\\
  \hline
 2& 24& 84 & 126 &84&$\bullet$ &$\bullet$&$\bullet$  \\
  \hline
 3 &0 & 0 &0 &$\star$ &$\bullet$&$\bullet$&4 \\
   \hline 
\end{tabular}

The Betti diagram of  $\Jcal_{[1,1,1,1]}$ :

\begin{tabular}{|c|c|c|c|c|c|c|c|c|c|c|c|}
  \hline
   & 0 & 1 & 2 &3&4&5&6&7&8&9&10\\
  \hline
 2& 55& 320 & 891 &1408&$\bullet$ &$\bullet$&$\star$&0&0&0&0  \\
  \hline
 3 &0 & 0 &0 &$\star$ &$\bullet$&$\bullet$&1408&891&320&55&0 \\
   \hline 
   4 &0 & 0 &0 &0 &0&0&0&0&0&0&1 \\
   \hline 
\end{tabular}

\end{example}
\begin{corollary}
If $n\geq 3$, the minimal free resolution of $\Jcal_{[{\bf b}]}$ is  pure if and only if $n=3$ and $b_1=b_2=b_3=1$.
\end{corollary}
\begin{proof} The proof  is an  immediate consequence of the Example \ref{example-fin}  and of Corollary \ref{betti}.
\end{proof} 

 \end{document}